\documentclass{amsart}

\usepackage[utf8]{inputenc}
\usepackage[T1]{fontenc}
\usepackage[french,english]{babel}
\usepackage{amsmath}
\usepackage{amssymb}
\usepackage{amsthm}
\usepackage{amsfonts}
\usepackage{csquotes}
\usepackage{multirow}

\usepackage[english]{babel}

\usepackage[letterpaper,top=2.5cm,bottom=2.5cm,left=3cm,right=3cm,marginparwidth=1.75cm]{geometry}

\usepackage{amssymb}
\usepackage[utf8]{inputenc}
\usepackage[english]{babel}
\usepackage{listings}
\usepackage{color}
\usepackage{setspace}
\usepackage{hyperref}
\usepackage{acro}
\usepackage{amsmath}
\usepackage{mathrsfs}
\usepackage{amsthm}
\usepackage{graphicx}
\usepackage{geometry}
\usepackage{subcaption}
\usepackage{cancel}
\usepackage{tikz}
\usepackage{color}
\usepackage{comment}
\usepackage{enumerate}
\usepackage{amsfonts}

\usepackage{tikz-cd}
\usepackage{mathtools}
\usepackage{MnSymbol}

\usepackage{xfrac}
\usepackage{faktor}
\usepackage{xcolor}
\usepackage{comment}
\usepackage{braket}
\usepackage{bbm}
\usepackage{array}
\usepackage{float}

\usepackage{lipsum}
\usepackage{blindtext}
\usepackage{subfiles}
\usepackage[titletoc,title]{appendix}
\usepackage[most]{tcolorbox}
\usepackage[new]{old-arrows}
\usepackage{tikz}
\usetikzlibrary {decorations.markings,arrows, decorations.pathmorphing, decorations.pathreplacing, decorations.shapes}
\usepackage[backend=biber, style=alphabetic, sorting=nyt]{biblatex}
\nocite{*}

\newcommand{\Q}{\mathbb{Q}}

\newcommand{\Z}{\mathbb{Z}}
\newcommand{\F}{\mathbb{F}}

\DeclareMathOperator{\characteristic}{char}

\DeclareMathOperator{\Disc}{Disc}
\DeclareMathOperator{\Jac}{Jac}
\DeclareMathOperator{\Gal}{Gal}

\DeclareMathOperator{\End}{End}

\DeclareMathOperator{\ord}{ord}

\DeclareMathOperator{\Tr}{Tr}

\renewcommand{\to}{\rightarrow}

\newcommand{\onto}{\twoheadrightarrow}
\newcommand{\longto}{\longrightarrow}

\newcommand{\iso}{\cong}
\newcommand{\lp}{\left(}
\newcommand{\rp}{\right)}

\newcommand{\<}{\left\langle}
\renewcommand{\>}{\right\rangle}
\newcommand{\vs}{\vspace{3mm}}
\newcommand{\svs}{\vspace{1.5 mm}}

\newtheorem{theorem}{Theorem}
\newtheorem{lemma}[theorem]{Lemma}
\newtheorem{proposition}[theorem]{Proposition}
\newtheorem{corollary}[theorem]{Corollary}

\theoremstyle{definition} 
\newtheorem{definition}[theorem]{Definition}
\newtheorem{example}[theorem]{Example}
\newtheorem{remark}[theorem]{Remark}
\newtheorem*{notation}{Notation}

\usepackage{fancyhdr,textcase}
\title{Distortion maps for elliptic curves over finite fields}

\author{Nikita Andrusov, Sevag Büyüksimkeşyan, Dimitrios Noulas, Fabien Pazuki,\\
Mustafa Umut Kazancıoğlu, Jordi Vilà-Casadevall}

\address{School of Applied Mathematics and Computer Science, Moscow Institute of Physics and Technology, 141701, Institutskiy per. 9, Dolgoprudny, Moscow Region, Russia.}
\email{andrusov.n@gmail.com}

\address{Mathematical Institute, University of Münster, Einsteinstr. 62, 48149 Münster, Germany.}
\email{sevag.buyuksimkesyan@uni-muenster.de}

\address{Department of Mathematics, National and Kapodistrian University of Athens, Panepistimioupolis, 15784 Athens, Greece.}
\email{dnoulas@math.uoa.gr}

\address{Department of Mathematical Sciences, University of Copenhagen,
Universitetsparken 5, 
2100 Copenhagen \O, Denmark.}
\email{fpazuki@math.ku.dk}

\address{Bernoulli Institute, Rijksuniversiteit Groningen, Nijenborgh 9, 9747 AG Groningen, The Netherlands
\\ 
Faculty of Engineering and Natural Sciences, 
Sabanc{\i} University, 
34956 \.Istanbul, Turkey.}
\email{m.u.kazancioglu@rug.nl, mustafa.kazancioglu@sabanciuniv.edu}

\address{School of Mathematics, University of Bristol, BS8 1UG Bristol, UK.}
\email{jordi.vilacasadevall@bristol.ac.uk}

\begin{document}

\begin{abstract}
The Weil pairing on elliptic curves has deep links with discrete logarithm problems, as can be seen, for instance, in \cite{Sch16, MOV93}. In practice, to better suit the functionalities of cryptosystems, one often needs to modify the original Weil pairing via what is called a distortion map. We propose a study on the question of the existence of distortion maps for elliptic curves over finite fields. We revisit results from the literature and provide detailed proofs. We also propose new perspectives at times.
\end{abstract}

\maketitle

{\flushleft
}


\section{Introduction}



Elliptic curves over finite fields are a cornerstone of modern public-key cryptography. Their main advantage over classical cryptographic systems, which use multiplicative groups of finite fields, lies in the apparent difficulty of the Elliptic Curve Discrete Logarithm Problem (DLP). Because no sub-exponential algorithm is known for solving the DLP in most cases, elliptic curve cryptography (ECC) achieves a comparable level of security with much smaller key sizes. This efficiency makes ECC particularly appealing for applications with limited computational or storage resources.\svs

Beyond the DLP, many cryptographic protocols, such as the Diffie-Hellman key exchange, rely on the hardness of the related Decision Diffie-Hellman (DDH) problem. For a cyclic group $G = \langle P \rangle$ of order $m$, the DDH problem is to computationally distinguish tuples of the form $(aP, bP, abP)$ from random tuples of the form $(aP, bP, cP)$. For an elliptic curve $E$ over a finite field $\mathbb{F}_q$, a powerful tool for analyzing the DDH problem is the Weil pairing. The Weil pairing $e_m: E[m] \times E[m] \to \mu_m$ is a bilinear, non-degenerate map, where $E[m]$ stands for the $m$-torsion subgroup of the rational points of the curve $E(\overline{\mathbb{F}}_q)$, and where $\mu_m$ is the group of $m$-th roots of unity.\svs

 A critical property of the Weil pairing is that it is alternating, meaning $e_m(Q, Q) = 1$ for any point $Q \in E[m]$. This property renders the Weil pairing thus useless for solving the DDH problem within the cyclic subgroup $\langle P \rangle$. We have $e_m(aP, bP) = e_m(P, P)^{ab}  = 1$ and $e_m(P, cP) = e_m(P, P)^c =  1$, so the pairing is incapable of distinguishing the valid Diffie-Hellman tuple. This limitation motivates the central question of this paper. The concept of a \textbf{distortion map} was introduced to overcome this problem. A distortion map is an endomorphism $\phi: E \to E$ with the property that for a point $P$ of order $m$, its image $\phi(P)$ does not lie in the cyclic subgroup generated by $P$; that is, $\phi(P) \notin \langle P \rangle$.\svs

The existence of such a map allows for the construction of a modified, non-degenerate pairing $e_{m,\phi}$ defined by $e_{m, \phi}(P, Q) = e_m(P, \phi(Q))$ for all $P,Q\in E[m]$. The bilinearity of the Weil pairing, combined with the fact that $\{P, \phi(P)\}$ can form a basis for $E[m]$, ensures that $e_{m, \phi}(P, P)$ is a primitive $m$-th root of unity. This modified pairing can be used to efficiently solve the DDH problem in $\langle P \rangle$ by simply checking if $e_{m, \phi}(aP, bP) \stackrel{?}{=} e_{m, \phi}(P, cP)$. Given their critical role in determining the hardness of the DDH problem, a complete characterization of when distortion maps exist is a fundamental question of both theoretical and practical importance.\svs

This paper provides cumulative information from the literature related to the existence of distortion maps and their properties. In addition to surveying known results, it provides explicit examples, clarified proofs, and new perspectives on several key phenomena; for instance, how distortion maps can be transferred via isogenies from one curve to another, as well as how to find a criteria for determining when a given isogeny is a distortion map.\svs

The existence of a distortion map is fundamentally an algebraic question about the structure of the curve's endomorphism ring, $\End(E)$. As such, it is linked with properties of its Frobenius endomorphism, $\pi_q: (x, y) \mapsto (x^q, y^q)$, and with classical notions from algebraic number theory: orders in quadratic number fields, quaternions, splitting properties of ideals, \textit{etc}. There is a nice connection between the study of $\End(E)$ and Ren\'e Schoof's work, we refer to \cite{Sch85, Sch93, Sch16} for some explicit links. In algorithmic number theory, we notably have Schoof's algorithm \cite{Sch93}. It is used to compute efficiently the number of points on an elliptic curve $E/\F_q$. Its primary computation is finding the trace of the Frobenius endomorphism, denoted $t$.  The number of points is then derived from this trace by Hasse's theorem, which reads $\# E(\F_q) = q + 1 - t$. In most successful implementations of protocols involving elliptic curves, being able to access the cardinality of $E(\mathbb{F}_q)$ is crucial.\svs

The cryptographic consequences of distortion maps are profound. On one hand, for any cryptographic system whose security relies on the hardness of the DDH problem in a subgroup $\langle P \rangle$, the existence of a distortion map for $\langle P \rangle$ is a weakness, as it renders the problem easier. For these applications, one must specifically choose curves that are guaranteed not to have distortion maps. On the other hand, the entire field of pairing-based cryptography, which enables advanced constructions such as identity-based encryption, short signatures, and attribute-based encryption, requires an efficiently computable, non-degenerate bilinear pairing. Distortion maps are one of the primary mechanisms used to construct these necessary pairings. Therefore, understanding the existence of distortion maps is essential for both the construction of cryptographic schemes and for cryptographic security analysis.\svs

Section \ref{Section 2} establishes the mathematical foundation by introducing distortion maps, alongside some of the key concepts for their study. It provides a criterion in Proposition \ref{Proposition Legendre symbol} for determining for which points a given isogeny is a distortion map, it details the properties of the modified Weil pairing in Proposition \ref{Properties modified Weil pairing} and it demonstrates how isogenies can be used to transfer distortion maps between elliptic curves in Proposition \ref{Prop transfer}. Following this, Section \ref{Section 3} investigates the case of ordinary elliptic curves, introducing Verheul's theorem and applying the criterion from the previous section to characterize the existence of distortion maps. Section \ref{Section 4} turns to the supersingular case, where Tate's theorem is used to prove the existence of distortion maps for $E[m]$. Section \ref{Section 5} reviews algorithmic constructions, specifically those relying on isogeny cycles and the constructive Deuring correspondence, while also summarizing known distortion maps for both supersingular and ordinary curves. Finally, Section \ref{Section 6} extends the framework to higher dimensions, examining briefly  distortion maps on jacobians of curves and supersingular abelian varieties.

\section*{Acknowledgments}
We thank \href{https://www.cimpa.info/en/home}{CIMPA} for organizing the summer school ``Elliptic curves and their applications'' in Yerevan during the period July 14-26, 2025. This is where we started working on the topic of distortion maps. We are grateful to the \href{http://math.sci.am/}{Institute of Mathematics of National Academy of Sciences of the Republic of Armenia} for hosting the school. Many thanks to Diana Davidova, Tigran Hakobyan, Valentijn Karemaker, Mihran Papikian, and also to Erenay Boyalı, Mahabba El Sahili, and Rahim Rahmati-Asghar, for nice mathematical discussions during the CIMPA school.\svs

J. Vilà-Casadevall was supported by a ``la Caixa'' Foundation Postgraduate Fellowship. F. Pazuki was supported by ANR {\it Jinvariant} (ANR-20-CE40-0003).


\section{Distortion maps and modified Weil pairing}\label{Section 2}

\subsection{Distortion maps}
Let $E$ be an elliptic curve defined over $\mathbb{F}_q$, where  $q = p^n$, $p$ is a prime and $n\geqslant1$ is an integer. We denote by $\End(E)$ the (geometric) endomorphism ring of $E$, \textit{i.e.}, the ring of isogenies $E \to E$ defined over $\overline{\F}_q$. This ring can have two possible structures:
\begin{itemize}
    \item $\End(E)$ is an order in an imaginary quadratic field.
    \item $\End(E)$ is an order in a definite quaternion algebra over $\Q$.
\end{itemize}

In the first case, we say that $E$ is \textit{ordinary}. In the second case, we say that $E$ is \textit{supersingular}. The trace of the $q$-th power Frobenius endomorphism is a multiple of $p$  if and only if $E$ is supersingular (see, for example, \cite[Proposition 4.31]{Was08}).

\begin{definition}
Let $P \in E \lp \overline{\F}_q \rp$ be a point of order $m > 1$ coprime to $p$. A \textit{distortion map for $P$} (or, equivalently, for the cyclic group $\< P \>$) is an isogeny $\phi \in \End(E)$ such that $\phi(P) \notin \< P \>$. A \textit{distortion map for $E[m]$} is an isogeny $\phi \in \End(E)$ such that $\phi(P) \notin \< P \>$ for every $P \in E[m] \setminus \{O\}$.
\end{definition}

As we will see in the following sections, the structure of the endomorphism ring of $E$ plays a crucial role when studying the existence of distortion maps on $E$. Another important notion when studying distortion maps for a point $P \in E(\F_q)$ of prime order $\ell$ is what we call the embedding degree.

\begin{definition} For a prime $\ell \neq p$, the \textit{embedding degree} $k(q,\ell)$ is defined to be the order of $q$ in the multiplicative group $(\Z/\ell\Z)^\times$.
\end{definition}

For an integer $m > 1$ coprime to $p$, we denote by $\mu_m$ the group of $m$-th roots of unity in $\overline{\F}_q^\times$ and by $e_m:E[m] \times E[m] \to \mu_m$ the Weil pairing on $E[m]$.

\begin{remark}
The embedding degree $k(q,\ell)$ is the smallest $k \geqslant 1$ for which $\mu_\ell \subseteq \F_{q^k}^\times$. If $E[\ell] \subseteq E(\F_q)$ for some prime $\ell \neq p$, since the Weil pairing is Galois-invariant, $\mu_\ell \subseteq \F_q^\times$, and hence $k(\ell,q) = 1$. On the other side, if $\ell \mid \# E(\F_q)$ and $k(\ell,q) > 1$, a theorem of Balasubramanian and Koblitz \cite[Theorem 1]{BK98} says that $k(\ell,q)$ is the minimal $k$ such that $E[\ell] \subseteq E(\F_{q^k})$.
\end{remark}

Any isogeny $\phi \in \End(E)$ induces a $\Z/\ell\Z$-linear endomorphism of $E[\ell]$ by restriction, which we denote by $\phi \vert_{E[\ell]}$. For a point $P \in E[\ell]$, the condition $\phi(P) \notin \<P\>$ is equivalent to saying that $P$ is not an eigenvector of $\phi \vert_{E[\ell]}$. Charles \cite{Cha06} used this idea and the ramification of primes in imaginary quadratic fields to study the existence of distortion maps for ordinary elliptic curves. His methods can be extended to the case of supersingular elliptic curves. To present these, we recall the definition of the Legendre symbol.

\begin{definition}
Let $D$ be an integer and $\ell$ a prime. For $\ell > 2$, we define the \textit{Legendre symbol} by
$$
\lp \frac{D}{\ell} \rp =
\left\{\begin{matrix*}[l]
\ 1 & \quad \text{if } D \text{ is a quadratic residue modulo } \ell, \text{ and } \ell \nmid D, \\
-1 & \quad \text{if } D \text{ is a quadratic non-residue modulo } \ell, \\
\ 0 & \quad \text{if } \ell \mid D.
\end{matrix*}\right.
$$

For $\ell = 2$, we set $\lp \frac{D}{2} \rp = 1$ for $D \equiv \pm 1 \textup{ mod } 8$, $\lp \frac{D}{2} \rp = -1$ for $D \equiv \pm 3 \textup{ mod } 8$, and $\lp \frac{D}{2} \rp = 0$ if $2 \mid D$.
\end{definition}

\begin{remark}
It is a well-known fact in algebraic number theory that, if $K$ is a quadratic imaginary number field with fundamental discriminant $D_K$, then a prime $\ell$ splits in $K$ if $\lp \frac{D_K}{\ell} \rp = 1$, ramifies in $K$ if $\lp \frac{D_K}{\ell} \rp = 0$, and remains inert in $K$ if $\lp \frac{D_K}{\ell} \rp = -1$. In particular, since it always holds that $D_K\equiv0,1 \mod{4}$, the prime $\ell=2$ is split if $D_K\equiv 1\mod{8}$, is inert if $D_K\equiv 5 \mod{8}$ and ramifies if $2 \mid D_K$.
\end{remark}

Let $B = \End(E)\otimes_\Z \Q$ be the endomorphism algebra of $E$, which is either an imaginary quadratic field or a definite quaternion algebra over $\Q$. In both cases, for a given isogeny $\phi \in \End(E) \setminus \Z$, one has that $\Z[\phi]\otimes_\Z \Q$ is imaginary quadratic subfield of $B$, which we denote by $K_\phi$. Then $\Z[\phi]$ is an order in $K_\phi$.

\begin{proposition}\label{Proposition Legendre symbol}
Let $\ell \neq p$ be a prime and $\phi \in \End(E) \setminus \Z$ be an isogeny. Let $D_\phi$ be the discriminant of $\Z[\phi] \subseteq K_\phi$. If $\ell \mid \left[ \End(E) \cap K_\phi : \Z[\phi] \right]$, then $\phi$ is not a distortion map for any $P \in E[\ell]$. Otherwise,
\begin{enumerate}
    \item $\lp \frac{D_\phi}{\ell} \rp = 1$ if and only if $\phi$ is a distortion map for all but two cyclic subgroups of $E[\ell]$.
    \item $\lp \frac{D_\phi}{\ell} \rp = -1$ if and only if $\phi$ is a distortion map for $E[\ell]$.
    \item $\lp \frac{D_\phi}{\ell} \rp = 0$ if and only if $\phi$ is a distortion map for all but one cyclic subgroup of $E[\ell]$.
\end{enumerate}
\end{proposition}

\begin{proof}
Let $\mathcal{O}_{K_\phi}$ be the ring of integers of $K_\phi$ and consider the inclusion of orders $\Z[\phi] \subseteq \End(E) \cap K_\phi \subseteq \mathcal{O}_{K_\phi}$. Let $f = \left[ \End(E) \cap K_\phi : \Z[\phi] \right]$ and $f' = \left[ \mathcal{O}_{K_\phi} : \End(E) \cap K_\phi \right]$. Then $ff' = \left[ \mathcal{O}_{K_\phi}:\Z[\phi] \right]$. We denote the fundamental discriminant of $\mathcal{O}_{K_\phi}$ by $D_{K_\phi}$, so $D_\phi = (ff')^2D_{K_\phi}$.\svs

Assume first that $\ell \mid f$ and let $\beta \in \End(E) \cap K_\phi$ be generator of this order, \textit{i.e.}, $\End(E) \cap K_\phi = \Z[\beta]$. Then, $\Z[\phi] = \Z[f\beta]$. Hence, after changing the sign of $\beta$ if necessary, we can write $\phi = a + f\beta$ for some $a \in \Z$. Since multiplication by $f$ induces the 0 map on $E[\ell]$, we have that $\phi \vert_{E[\ell]} = a\vert_{E[\ell]}$, so $\phi$ is not a distortion map for any $P \in E[\ell]$.\svs

 From now on, we assume that $\ell \nmid f$. The rest of the proof relies on the splitting properties of the characteristic polynomial of $\phi$, which we denote by $\chi_\phi$. We link these properties with the distorsion property on the one hand, and link the same properties with the Legendre symbol on the other hand, which will be enough to conclude.
 
 Observe that the characteristic polynomial of $\phi\vert_{E[\ell]}$ as a linear map on $E[\ell]$ is the reduction $\bar{\chi}_\phi$ of $\chi_\phi$ modulo $\ell$. Three possible situations may occur:
\begin{enumerate}[(i)]
    \item $\bar{\chi}_\phi(x) \in \F_\ell[x]$ is irreducible,
    \item $\bar{\chi}_\phi(x) = (x-\bar{a})(x-\bar{b})$ for two different $\bar{a},\bar{b} \in \F_\ell$,
    \item $\bar{\chi}_\phi(x) = (x-\bar{a})^2$ for some $\bar{a} \in \F_\ell$.
\end{enumerate}

In situation (i), since $\bar{\chi}_\phi$ doesn't have roots in $\F_\ell$, the map $\phi\vert_{E[\ell]}$ doesn't have eigenvectors, so $\phi$ is a distortion map for $E[\ell]$. In situation (ii), $\phi\vert_{E[\ell]}$ has two different eigenvalues, and hence two different one-dimensional eigenspaces. Thus $\phi$ will be a distortion map for a point $P$ if, and only if, $P$ doesn't lie in any of the eigenspaces. In situation (iii), $\phi\vert_{E[\ell]}$ has a unique eigenvalue. The corresponding eigenspace could be 1 or 2-dimensional. We claim that it is necessarily 1-dimensional. Assume otherwise. Then $\phi\vert_{E[\ell]} = a\vert_{E[\ell]}$ for some integer $a \in \Z \subseteq \End(E)$. This implies that $\ker(\ell) \subseteq \ker(\phi - a)$. Since multiplication by $\ell$ is a separable isogeny, as $\ell \neq p$, there exist some $\psi \in \End(E)$ with $\ell \psi = \phi - a$ (see \cite[Corollary III.4.11]{Sil09}). As we have $\Z[\phi] = \Z[\ell\psi]$, we get $\Z[\psi] \otimes_\Z \Q = \Z[\phi] \otimes_\Z \Q = K_\phi$, and hence $\phi \in \End(E) \cap K_\phi$. Moreover, $\ell =\left[ \Z[\psi]:\Z[\phi] \right] \mid f$, which contradicts our assumptions. This proves the claim. In situation (iii), $\phi$ will be a distortion map for all points $P $ not lying in the one-dimensional eigenspace of $\phi\vert_{E[\ell]}$.\svs

Consider now as a first subcase that $\ell \mid f'$, so $\ell \mid ff'$. Let $\alpha \in \mathcal{O}_{K_\phi}$ be a generator of the ring of integers, \textit{i.e.}, $\mathcal{O}_{K_\phi} = \Z[\alpha]$. Changing the sign of $\alpha$ if necessary, we can write $\phi = a + ff'\alpha$ for some $a \in \Z$. Let $\chi_\alpha(x) = x^2 + bx + c$, with $b,c \in \Z$, be the characteristic polynomial of $\alpha$. Then $\chi_\phi(x) = (x-a)^2 +bff'(x-a) + c(ff')^2$, and hence $\bar{\chi}_\phi(x) = (x-\bar{a})^2$, \textit{i.e.}, we are in situation (iii). Since $\ell \mid D_\phi = (ff')^2D_K$, this happens precisely when $\lp \frac{D_\phi}{\ell} \rp = 0$.\svs

Consider now that $\ell \nmid f'$, so $\ell \nmid ff'$. By the Dedekind–Kummer theorem (see \cite{Conrad}, Theorem 1), situation (i) will occur precisely when $\ell$ remains inert in $K_\phi$, (ii) will happen precisely when $\ell$ splits in $K_\phi$, and (iii), when $\ell$ ramifies in $K_\phi$. The first case takes place precisely when $\lp \frac{D_K}{\ell} \rp = -1$ and, since $D_\phi$ and $D_K$ differ by a non-zero square modulo $\ell$, this is equivalent to saying that $\lp \frac{D_\phi}{\ell} \rp = -1$. The second case occurs precisely when $\lp \frac{D_K}{\ell} \rp = 1$. Again, this is equivalent to saying that $\lp \frac{D_\phi}{\ell} \rp = 1$. Finally, the third case occurs precisely when $\ell \mid D_K$, which is equivalent to $\ell \mid D_\phi$ and $\lp \frac{D_\phi}{\ell} \rp = 0$.
\end{proof}

One can derive several interesting consequences from the previous proposition. The first one allows us to determine when a given isogeny will be a distortion map for $E[m]$ for some integer $m$. Before stating it, we first show the following fact:

\begin{lemma}
Let $m > 1$ be an integer coprime to $p$. An isogeny $\phi$ is a distortion map for $E[m]$ if and only if $\phi$ is a distortion map for $E[\ell]$ for all primes $\ell \mid m$. In particular, $\phi$ is a distortion map for $E[\ell]$ if and only if it is for $E[\ell^n]$, for all integers $n>1$.
\end{lemma}

\begin{proof}
The direct implication is trivial as $E[\ell] \subseteq E[m]$ for all $\ell \mid m$. We prove the inverse implication by contrapositive. Assume $\phi$ is not a distortion map for $E[m]$ and let $P \in E[m]$ be a non-trivial point with $\phi(P) = aP$ for some $a \in \Z$. Let $\ell$ be a prime with $\ell \mid \ord(P)$. Then $Q := \frac{\ord(P)}{\ell} P$ is a point of order $\ell$. Moreover, 
$$
\phi(Q) = \phi \lp \frac{\ord(P)}{\ell} P \rp = \frac{\ord(P)}{\ell} a P = aQ,
$$
so $\phi$ is not a distortion map for $E[\ell]$.
\end{proof}

\begin{corollary}\label{Corollary distortion E[m]}
Let $m$ be an integer, $\phi \in \End(E) \setminus \Z$ and $D_\phi = \textup{disc}(\Z[\phi])$. Then, $\phi$ is a distortion map for $E[m]$ if and only if $\lp \frac{D_\phi}{\ell} \rp = -1$ for all $\ell \mid m$.
\end{corollary}

Another consequence of Proposition \ref{Proposition Legendre symbol} is that there is no universal distortion map:

\begin{corollary}
There is no $\phi \in \End(E)$ that is a distortion map for every point of $E$ of order coprime to $p$. 
\end{corollary}
\begin{proof}
Fix $\phi \in \End(E) \setminus \Z$. For a given $\ell \neq p$, Proposition \ref{Proposition Legendre symbol} tells us that $\phi$ will be a distortion map for $E[\ell]$ if and only if $\lp \frac{D_\phi}{\ell} \rp = -1$. Since $D_\phi < -1$, there is at least one prime $\ell$ with $\lp \frac{D_\phi}{\ell} \rp = 0$, which is enough to conclude. We may add furthermore that by the Chebotarev density theorem, $\lp \frac{D_\phi}{\ell} \rp = 1$ for $50\%$ of the primes $\ell$ (when ordered naturally).
\end{proof}

\subsection{The trace map}

We saw in the previous section how to determine whether a given isogeny is a distortion map. For any elliptic curve $E \, / \, \F_q$, one can always consider the $q$-th power Frobenius endomorphism $\pi_q$. Another isogeny that is sometimes used as a distortion map is the trace map, which we now define. In this section, we fix a prime $\ell \neq p$ with $\ell \mid \# E(\F_q)$ and denote the embedding degree $k(q,\ell)$ by $k$.

 \begin{definition}
 Let $k\geqslant 1$ be an integer. For a point $P \in E(\F_{q^k})$, one defines the \textit{trace of} $P$ (relative to $k$) as
$$
\Tr_k(P) = \sum_{i = 0}^{k-1} \pi_q^i(P). 
$$
 \end{definition}

The trace map $\Tr_k$ is an isogeny. We shall see that, when $k > 1$, the trace map allows us to construct a distortion map for certain points.

\begin{lemma}
Assume $k > 1$. Then $\pi_q \mid_{E[\ell]}$ has two different eigenspaces, with respective associated eigenvalues equal to $1$ and $q$.
\end{lemma}
\begin{proof}
    As $\ell$ is a prime dividing $\# E(\F_q)$ by assumption, $E$ has some $\ell$-torsion point defined over $\F_q$, and as such a point is fixed by $\pi_q$, the map $\pi_q \vert_{E[\ell]}$ must have an eigenvector associated with the eigenvalue $1$. Let $t = q + 1 - \# E(\F_q)$ denote the trace of $\pi_q$. Then, the characteristic polynomial of $\pi_q \vert_{E[\ell]}$ is the reduction modulo $\ell$ of the characteristic polynomial of $\pi_q$, and hence given by $X^2 - \bar{t}X + \bar{q}$. As $t \equiv q+1 \mod{\ell}$, this polynomial factors as $(X - \bar{1})(X - \bar{q})$. Observe that $\bar{1}$ and $\bar{q}$ are different if and only if $q \not\equiv 1 \mod{\ell}$. This is equivalent to saying that $k > 1$ by definition of the embedding degree.
\end{proof}

The eigenspace associated with the eigenvalue $1$ is $E(\F_q)[\ell]$. The points $P \in E(\F_q)[\ell]$ satisfy that $\Tr_k(P) = kP$, and hence $\Tr_k$ is not a distortion map for these points. Boneh showed that the eigenspace associated with the eigenvalue $q$ is precisely the set of points $P \in E[\ell]$ of trace $0$ (see \cite[Lemma 3]{BLS04} or \cite[Lemma IX.16]{Gal09}). Hence, the trace map is not a distortion map for this eigenspace either. The following lemma shows that the trace map is a distortion map for all other points:

\begin{lemma}
Assume $k > 1$. Let $P \in E[\ell]$ be a point not lying in one of the eigenspaces of $\pi_q \mid_{E[\ell]}$. Then, $\Tr_k$ is a distortion map for $P$.
\end{lemma}

\begin{proof}
Observe that, since $k$ is the order of $q$ in $(\Z/\ell\Z)^\times$, we have $k \mid \ell-1$. In particular, $\ell \nmid k$. Thus, the action of the trace map on $E[\ell]$ has two different eigenvalues: $k$ and $0$. As explained before, the eigenspaces of $\Tr_k \vert_{E[\ell]}$ of the eigenvalues $k$ and $0$ coincide, respectively, with the eigenspaces of $\pi_q \vert_{E[\ell]}$ of eigenvalues 1 and $q$. Hence, if $P$ does not lie in an eigenspace of $\pi_q \vert_{E[\ell]}$, it does not lie in an eigenspace of $\Tr_k \vert_{E[\ell]}$ either. The result follows.
\end{proof}

\subsection{Modified Weil pairing}

The interest of distortion maps relies on the fact that they allow us to modify the usual Weil pairing on $E[m]$ to obtain a new pairing that is not skew-symmetric, but preserves all the other interesting properties of $e_m$. To describe it, let us fix an integer $m \in \Z_{>0}$ coprime to $p$ and a distortion map $\phi \in \End(E)$ for $E[m]$.

\begin{definition}
The \textit{modified Weil pairing} with respect to $\phi$ on $E[m]$ is the pairing
$$
\begin{matrix}
e_{m,\phi}: & E[m] \times E[m] & \longrightarrow & \mu_m \\
& (P,Q) & \longmapsto & e_m(P, \phi(Q))
\end{matrix}
$$
where $e_m$ is the Weil pairing on the $m$-th torsion points.
\end{definition}

\begin{proposition}\label{Properties modified Weil pairing}
The modified Weil pairing satisfies the following properties:
\begin{enumerate}
    \item Bilinearity: for all $P_1, P_2, Q, P, Q_1, Q_2 \in E[m]$,
    $$e_{m, \phi}(P_1 + P_2,Q) = e_{m,\phi}(P_1,Q)e_{m,\phi}(P_2,Q),$$
    $$e_{m, \phi}(P,Q_1+Q_2) = e_{m,\phi}(P,Q_1)e_{m,\phi}(P,Q_2).$$
    
    \item Distortion: let $P\in E[m]$ be a point of order $m$. Then,
    $$\ord_{\mu_m} \big(e_{m,\phi}(P,P) \big) = m.$$ 
    
    \item Nondegeneracy: Fix $Q \in E[m]$. 
    \begin{center}
        If $e_{m,\phi}(P,Q) = 1$ for all $P \in E[m]$, then $Q = O$.
        
        If $e_{m,\phi}(Q,P) = 1$ for all $P \in E[m]$, then $Q = O$.
    \end{center}
    
    \item Compatibility: if $m = d d'$, then
    \begin{center}
        $e_{m,\phi}(P,Q) = e_{d,\phi}(d'P,Q)$ for all $P \in E[m]$ and all $Q \in E[d]$.
    \end{center}
    \begin{center}
        $e_{m,\phi}(P,Q) = e_{d,\phi}(P,d'Q)$ for all $P \in E[d]$ and all $Q \in E[m]$.
    \end{center}
    \item Galois invariance: assume $\phi$ is defined over $\F_{q^k}$. For all $P, Q \in E[m]$,
    \begin{center}
        $e_{m,\phi}(P^\sigma,Q^\sigma) = e_{m,\phi}(P,Q)^\sigma$ for all $\sigma \in \Gal(\overline{\F}_q / \F_{q^k})$.
        \end{center}
\end{enumerate}
\end{proposition}
\begin{proof}
Properties (1) and (4) follow from the linearity of $\phi$ and the analogous properties for the Weil pairing (see \cite[Proposition III.8.1]{Sil09}). Property (5) follows from the fact that, if $\sigma \in \Gal(\overline{\F}_q / \F_{q^k})$, then $\phi(P^\sigma) = \phi(P)^\sigma$, and the corresponding property of $e_m$.\svs
    
To show property (3), observe that $\phi \vert_{E[m]}$ is an isomorphism. Indeed, since $E[m]$ is finite, it suffices to show injectivity. By definition of a distortion map, if $P \neq O$, then $\phi(P) \notin \<P\>$. In particular, $\ker \phi \vert_{E[m]} = \{O\}$. Property (3) then follows from the analogous property of the Weil pairing.\svs

Finally, to prove (2), let $P$ be a point of order $m$ and let $Q$ be a point such that $P, Q$ form a basis of the $\mathbb{Z}$-module $E[m]$. If $e_m(P, \phi(P))$ is of order $k$ for some $k<m$, then $\phi(P) = aP + \frac{bm}{k}Q$ for some $a, b \in \mathbb{Z}/m\mathbb{Z}$. Then $\phi(kP) \in \langle P\rangle$ and $kP \neq O$ which contradicts the distortion property of $\phi$. Therefore $e_m(P, \phi(P))$ is of order $m$.
\end{proof}

When $m$ is square-free, the following stronger version of the distortion property for the modified Weil pairing holds:

\begin{corollary}[Strong distortion]\label{Strong distortion property}
Assume $m$ is square-free, let $n\mid m$, and let $P \in E[m]$ be a point of order $n$. Then,
    $$\ord_{\mu_m} \big(e_{m,\phi}(P,P) \big) = n.$$
\end{corollary}

\begin{proof}
Write $m = n n'$. By Proposition \ref{Properties modified Weil pairing}.(4), we have that
$$
e_{m,\phi}(P,P) = e_{n, \phi}(n'P, P) = e_{n,\phi}(P,P)^{n'}.
$$
By Proposition \ref{Properties modified Weil pairing}.(2), $e_{n,\phi}(P,P)$ is a primitive $n$-th root of unity. Moreover, since $m$ is square-free, $n$ and $n'$ are coprime. The result follows.
\end{proof}

\begin{remark}
Corollary \ref{Strong distortion property} fails if $m$ is not square-free. Indeed, take $m = \ell^2 m'$ and $P \in E[m]$ a point of order $\ell$. It follows from the compatibility property that $e_{m,\phi}(P,P) = 1$, even when $\phi(P)\notin \langle P\rangle$.
\end{remark}

\subsection{Isogeny transfer}
One can use isogenies to transfer distortion maps from one elliptic curve to another. Consider two elliptic curves $E_1, E_2$ over $\F_q$, and an isogeny $\psi: E_1 \to E_2$ of degree $d = \deg(\psi)$. We denote by $\widehat{\psi}: E_2 \to E_1$ the dual isogeny of $\psi$.

\begin{proposition}\label{Prop transfer}
Let $P \in E_1(\overline{\F}_q)$ be a point of order $m$ coprime to $pd$, and assume $\phi \in \End(E_1)$ is a distortion map for $P$. Then $\psi \circ \phi \circ \widehat{\psi} \in \End(E_2)$ is a distortion map for $\psi(P) \in E_2(\overline{\F}_q)$.
\end{proposition}

\begin{proof}
Let $e_m^1$ and $e_m^2$ denote the Weil pairings on $E_1[m]$ and $E_2[m]$, respectively. By the properties of the Weil pairing, it suffices to show that $e_m^2 \lp \psi(P), \psi \circ \phi \circ \widehat{\psi}(\psi(P)) \rp \neq 1$. Using \cite[Proposition III.8.2]{Sil09}, we have that 
$$
e_m^2 \lp \psi(P), (\psi \circ \phi \circ \widehat{\psi} \circ \psi)(P)) \rp = e_m^1 \lp (\widehat{\psi} \circ \psi) (P), (\phi \circ \widehat{\psi} \circ \psi) (P) \rp = e_m^1(P, \phi(P))^{d^2}.
$$

Since $d$ is coprime to $m$ and $e_m^1(P, \phi(P)) \neq 1$ by Proposition \ref{Properties modified Weil pairing}.(2), the result follows.
\end{proof}

\begin{corollary}\label{Cor transfer}
Let $m > 1$ be an integer coprime to $pd$ and assume $\phi \in \End(E_1)$ is a distortion map for $E_1[m]$. Then $\psi \circ \phi \circ \widehat{\psi} \in \End(E_2)$ is a distortion map for $E_2[m]$.
\end{corollary}

\begin{proof}
Since $d$ is coprime to $m$, $\psi$ induces an isomorphism from $E_1[m]$ to $E_2[m]$. The result follows from the previous proposition.
\end{proof}

\begin{remark}
One can apply the previous results to the dual isogeny to show that these are necessary and sufficient conditions.
\end{remark}

The endomorphism rings $\End(E_1)$ and $\End(E_2)$ are orders of the same endomorphism algebra. Moreover, they satisfy  $\psi \mathrm{End}(E_1) \widehat\psi \subseteq \mathrm{End}(E_2)$ and $\widehat\psi \mathrm{End}(E_2) \psi \subseteq \mathrm{End}(E_1)$ (see \cite[pg. 7]{Koh96}). Hence, if we denote by $\mathcal{D}_{E,P}$ and $\mathcal{D}_{E,m}$ the set of distortion maps for $P \in E(\bar{\F}_q)$ and $E[m]$ respectively, we have the following result:

\begin{corollary}
Let $m > 1$ be an integer coprime to $pd$, and $P \in E_1(\bar{\F}_q)$ and $Q \in E_2(\bar{\F}_q)$ points of order $m$. Then,
\begin{itemize}
\item $\psi \mathcal{D}_{E_1,P} \widehat{\psi} \subseteq \mathcal{D}_{E_2,\psi(P)}$ and $\widehat{\psi} \mathcal{D}_{E_2,Q} \psi \subseteq \mathcal{D}_{E_1,\widehat{\psi}(Q)}$.

\item $\psi \mathcal{D}_{E_1,m} \widehat{\psi} \subseteq \mathcal{D}_{E_2,m}$ and $\widehat{\psi} \mathcal{D}_{E_2,m} \psi \subseteq \mathcal{D}_{E_1,m}$.
\end{itemize}
\end{corollary}

\begin{example}\label{example}
By an implementation in SageMath\footnote{Code for all the examples: \url{https://github.com/noulasd/distortion}} \cite{Sagemath}, we will transfer via isogenies the distortion map from the third example in Table \ref{table1}. Let $\F_{101^6} = \F_{101}[X]/(X^6 + 2X^4 + 90X^3 + 20X^2 + 67X + 2)$ and $z$ the image of $X$ in $\F_{101^6}$. Let $\tau$ be a root of the polynomial $x^2-2$, which is not in $\mathbb{F}_{101}$. Consider the elliptic curve
\[E: \ y^2 = x^3 + \tau + 1 \, / \, \F_{101^2}.\]
Let $r$ be a root of the polynomial $x^2-(\tau+1)$ and $\omega$ a root of $x^3-r$ as specified in Table \ref{table1}. For instance, in $\F_{101^6} = \mathbb{F}_{101}(z)$, one can take:
\begin{align*}
    \tau &= 56z^5 + 25z^4 + 79z^3 + 20z^2 + 70z + 9, \\
    r & = 69z^5 + 29z^4 + 27z^3 + 3z^2 + 61z + 3, \\
    \omega & = 38z^5 + 90z^4 + 19z^3 + 94z^2 + 100z + 69.
\end{align*}
    
One has that $\#E(\mathbb{F}_{101^6}) = 2^2\cdot 3^4\cdot 7^2 \cdot 13^2 \cdot 17^2 \cdot 37^2$. We will pick a point $Q \in E[17]$ which, by Vélu's method \cite{Vel71} will yield an isogeny of degree $17$, then transfer the properly distorted points of $E[7]$. Note that, it is specified in \cite{JN03} that the map $\phi(x,y) = \left(\omega\frac{x^{p}}{r^{(2p-1)/3}},\frac{y^{p}}{r^{p-1}}\right)$ is a distortion map for the $6$ non-trivial points in $E[7](\mathbb{F}_{101^2})$. In fact, we computed that $\phi$ distorts all but $12$ non-trivial points in $E[7]$.\svs

Pick $Q = (32z^5 + 16z^4 + 91z^3 + 59z^2 + 49z, 90z^5 + 51z^4 + 44z^3 + 61z^2 + 62z + 41)$. This gives a 17-degree isogeny $\psi: E \rightarrow E^\prime$, where $E^\prime$ is given by the equation $y^2 = x^3 + (70z^5+4z^4+88z^3+z^2+12z+43)x + (6z^5+64z^4+77z^3+31z^2+58z+88)$. The new distortion map for $ E^\prime$ is 
\[\psi \circ \phi \circ \widehat\psi(x,y) = (g_1(x,y), g_2(x,y)),\]
where $g_1(x,y), g_2(x,y)$ are too large to be included in the text, but both are easily computed and have leading terms $(32z^5 + 32z^4 - 49z^3 - 50z^2 + 18z + 14)x^{29189}$ and $(-36z^5 + 20z^4 + 43z^3 + 16z^2 - 45z - 10)x^{43632}y^{101} $, respectively.\svs

We check that $\psi\circ \phi \circ \widehat \psi$ distorts all but $12$ non-trivial points of $E^\prime[7]$. For instance, let $P = (83z^5 + 30z^4 + 59z^3 + 26z^2 + 46z + 100, 36z^5 + 81z^4 + 56z^3 + 23z^2 + 32z + 56 ) \in E^\prime[7]$, then the modified Weil pairing of $E^\prime[7]$ with respect to $\psi\circ\phi\circ\widehat\psi$ maps $P$ to the $7$-th root of unity 
$52z^5 + 11z^4 + 12z^3 + 57z^2 + 81z + 72$. However, $P^\prime = (23z^4 + 33z^3 + 89z^2 + 66z + 46, 22z^5 + 30z^4 + 82z^3 + 4z^2 + 75z + 42) \in E^\prime[7] $ is sent to $1$.
\end{example}

\begin{remark}
In the previous example, we verified computationally that $\phi$ and $\psi \circ \phi \circ \widehat{\psi}$ distort all but 2 cyclic subgroups of $E[7]$ and $E'[7]$, respectively. This is consistent with the fact that $\phi \vert_{E[7]}$ has two distinct eigenvalues, which we computed to be $2\ \mathrm{mod}\,7$ and $5\ \mathrm{mod}\,7$. Indeed, $\phi^2(x,y) = \left(\zeta \cdot  x^{101^2}, y^{101^2}\right)$, where $\zeta = r^{-2(101^2-1)/3}$. The element $\zeta$ is a primitive cubic root of unity, as $r \in \mathbb{F}_{101^2}$ and is not a cube. We verify that $\phi^2 = [\zeta]\circ \pi_{101^2}$ acts on $E[7]$ as multiplication by 4, yielding the characteristic polynomial $X^2-4 \equiv (X-2)(X-5)\ \mathrm{mod}\,7$.
\end{remark}


\section{Distortion maps for ordinary elliptic curves} \label{Section 3}

In this section we assume that $E$ is an ordinary elliptic curve defined over $\F_q$ and $\ell \neq p$ is a prime. In this situation, all the endomorphisms of $E$ are defined over $\F_q$ (see, for example, \cite[Corollary 2.2]{Wit01}). The following theorem of Verheul shows that points with an embedding degree greater than 1 do not admit a distortion map.

\begin{theorem}[\cite{Ver04}, Theorem 6]\label{Thm no distortion map}
Assume $\ell \mid \# E(\F_q)$ and let $P \in E(\F_q)$ be a point of order $\ell$. If $E[\ell]$ is not defined over $\F_q$ (in particular, if the embedding degree $k(q,\ell) > 1$), then there exist no distortion maps for $P$. 
\end{theorem}

\begin{proof}
We have that $E[\ell]\cap E(\mathbb{F}_q) = \< P \>$. Let $\phi \in \End(E)$ be any isogeny, and $\pi_q \in \End(E)$ be the $q$-th power Frobenius. Recall that a point $Q$ is defined over $\F_q$ if and only if it is fixed by $\pi_q$. Observe that $\phi(P) \in E[\ell]$ and, since $\End(E)$ is commutative,
$$
\pi_q(\phi(P)) = \phi(\pi_q(P))= \phi(P),
$$
so $\phi(P) \in \< P \>$. Hence, $\phi$ is not a distortion map for $P$.
\end{proof}

Since $\End(E)$ is an order in a quadratic imaginary number field, one can apply Proposition \ref{Proposition Legendre symbol} to its generator to determine the existence of a distortion map for different points of $E$:

\begin{theorem}\label{Thm Charles}
Let $D = \Disc(\End(E))$. Then,
\begin{enumerate}
    \item If $\lp \frac{D}{\ell} \rp = 1$, all but two cyclic subgroups of $E[\ell]$ admit a distortion map.
    \item If $\lp \frac{D}{\ell} \rp = -1$, then $E[\ell]$ admits a distortion map.
    \item If $\lp \frac{D}{\ell} \rp = 0$, then all but one cyclic subgroup of $E[\ell]$ admit a distortion map.
\end{enumerate}
\end{theorem}

\begin{proof}
Let $\omega$ be a generator of $\End(E)$, \textit{i.e.}, $\End(E) = \Z[\omega]$. Then any isogeny $\phi \in \End(E)$ can be written as $\phi = a + b\omega$ for some $a,b \in \Z$. For a given point $P \in E[\ell]$, if $\phi$ is a distortion map for $P$, then so is $\omega$. Indeed by contrapositive, one clearly sees that if $\omega(P) \in \<P\>$, then $\phi(P) \in \<P\>$. Hence, a given point $P$ admits a distortion map if, and only if, $\omega$ is a distortion map for $P$.\svs

Let $K = \End(E) \otimes \Q$ be the endomorphism algebra of $\omega$. Then, using the notation of Proposition \ref{Proposition Legendre symbol}, we have that $K_\omega = K$ and $D_\omega = D$. Since $\left[\End(E) \cap K_\omega : \Z[\omega] \right] = 1$ in this setting, the result follows from Proposition \ref{Proposition Legendre symbol}.
\end{proof}

\begin{remark}
Charles \cite[Examples 3.1, 3.2, 3.3 and 3.4]{Cha06} gave examples satisfying the different cases from the previous theorem. Theorem \ref{Thm Charles} is a reformulation of Charles' \cite[Theorem 2.1]{Cha06}. However, the statement is slightly different from Charles', as we corrected some imprecision: Charles' Theorem 2.1 states that there are no points of $E[\ell]$ that admit a distortion map when $\ell \mid [\mathcal{O}_K : \End(E)]$. Theorem \ref{Thm Charles} shows that, when $\ell \mid [\mathcal{O}_K : \End(E)]$, all but one cyclic subgroup of $E[\ell]$ admit a distortion map. Let us now provide a computational example, extending the idea of Example 3.4 in Charles', where this case has been tested.
\end{remark}

\begin{example}[A case where $\ell \mid {[\mathcal{O}_K:\mathcal{O}]}$] \label{example_conductor}

    Let $\mathbb{F}_{13^3} = \mathbb{F}_{13}[X]/(X^3 + 2X + 11)$ and consider the curve $E: y^2 = x^3 + x + 5 \,/ \, \mathbb{F}_{13^3}$. Using SageMath, we compute that $\#E(\mathbb{F}_{13^3}) =2^2 \cdot 3^4 \cdot 7 $, and therefore the trace of the $13^3$-th power Frobenius $\pi_{13^3}$ is $t=-70$ and the order $\Z[\pi_{13^3}]$ has discriminant $D_{\pi_{13^3}} = - 3888$. Observe that $E$ is ordinary. Observe also that $K = \Z[\pi_{13^3}] \otimes \Q = \Q(\sqrt{-3})$, which has fundamental discriminant $D_K = -3$.
    We will argue that the conductor $[\mathcal{O}_K:\End(E)]$ is $3$ as follows. Let $E^\prime/\mathbb{Q}$ be given by the equation
    \[E^\prime : y^2 = x^3 - \frac{120000000}{64009}x + \frac{2000000000}{64009},\] which reduces to $E$ mod $13$. It holds that the $j$-invariant $j(E^\prime) = -12288000$ is a root of the Hilbert class polynomial $H_{-27} = H_\mathcal{O}$, for $\mathcal{O}$ being the order of conductor $3$ in the maximal order $\Z[(1+\sqrt{-3})/2]$, as $-27 = 3^2 \cdot D_K$. Therefore $\mathcal{O} \cong \mathrm{End}(E^\prime)$. We refer to \cite{Sut23} for the theory regarding the Hilbert class polynomial. Since $13$ splits in $K$, the Deuring reduction theorem \cite[Ch. 13, Thm. 12]{Lan87} asserts that $\mathrm{End}(E) \cong \mathrm{End}(E^\prime) \cong \mathcal{O}$.\svs

We will now construct the distortion map. Pick the $9$-torsion point $(3,3)$ of $E$, which by V\'elu's method \cite{Vel71}, induces a $9$-degree isogeny $\phi: E\rightarrow E$, given by the rational map
\[\phi(x,y) = \left(\frac{f_1(x)}{f_2(x)},  \frac{g_1(x,y)}{g_2(x,y)}\right),\] where 
\begin{align*}
f_1(x)&= x^9 + x^8 + 4x^7 + x^6 - 6x^5 - x^4 + 5x^3 - 6x^2 - 5, \\
f_2(x)&= x^8 + x^7 + 4x^6 + x^5 - 6x^4 - x^3 + 3x^2 + 5x + 4, \\
g_1(x,y)&= x^{12}y - 5x^{11}y - 5x^{10}y + 2x^9y + 2x^8y - x^7y - 2x^6y + 5x^4y - 4x^3y - 2x^2y + 3xy - 6y, \\
g_2(x,y)&= x^{12} - 5x^{11} - 5x^{10} + 2x^9 + 2x^8 - x^7 - 5x^6 + 2x^5 + x^4 - 5x^3 + 5x^2 - 2x + 5.
\end{align*}

Let $P=(12,4)$ be a $3$-torsion point which is defined over $\mathbb{F}_{13}$, then it holds that $E[3]\cap E(\mathbb{F}_{13}) = \langle P \rangle$. We verify that $\phi$ maps $\langle P \rangle$ to the point at infinity of $E$, while it distorts the points $Q\in E[3]\setminus E(\mathbb{F}_{13})$, as we compute that $\phi(Q)\in \langle P \rangle$. For $\ell=3$, this example showcases the $\mathbb{F}_\ell[x]/(x^2)$ structure of the action of $\mathrm{End}(E)$ on $E[\ell]$, as was determined in the proof of Proposition \ref{Proposition Legendre symbol} in the case of $\ell$ dividing the conductor $f^\prime$.
\end{example}

\begin{remark}
Since $E$ is ordinary, the trace $t$ of $\pi_q$ is not divisible by $p$, and in particular is non-zero. Hence, Hasse's theorem tells us that $t^2-4q < 0$, which implies that $\pi_q \notin \Z$. Thus, one can apply Corollary \ref{Corollary distortion E[m]} to find a suitable $m$ for which $\pi_q$ is a distortion map for $E[m]$. We illustrate this with an example.
\end{remark}

\begin{example}
Let $\F_{361} = \F_{19^2} := \F_{19}[X]/(X^2+18X+2)$ and let $a$ be the image of $X$ in $\F_{361}$. Consider the elliptic curve $E : y^2 = x^3 + a \ / \, \F_{361}$. Let $\pi_{361}$ be the $361$-st power Frobenius endomorphism. Using a computer algebra system, one deduces that $\# E(\F_{361}) = 325$. Its norm is $n = 361$ and its trace is $t = 361 + 1 - \# E(\F_{19^2}) = 37$. Observe that $E$ is ordinary. The discriminant of $\Z[\pi_{361}]$ is $D_{\pi_{361}} = t^2 - 4n = -75$. The smallest two primes $\ell$ with $\lp \frac{-75}{\ell} \rp = -1$ are 2 and $11$. Hence, $\pi_{361}$ is a distortion map for $E[22]$.\svs

One can computationally check that the smallest $n$ for which $\# E(\F_{361^n})$ is divisible by 22 is $60$, so $\F_{19^{120}}$ is the smallest field over which there is a point $P$ of order 22. By Theorem \ref{Thm no distortion map} we know that, in fact, $E[22]$ must be defined over $\F_{19^{120}}$. Indeed, since $E[2]$ admits a distortion map, the smallest field $\F_q$ with a 2-torsion point must also be the minimal field of definition of $E[2]$. The same reasoning holds for $E[11]$. Hence, both $E[2]$ and $E[11]$ are defined over $\F_{19^{120}}$, and thus so is $E[22]$.
\end{example}

\begin{example}[Isogeny transfer of \cite{Cha06}, Example 3.2] 
The ordinary elliptic curve in question is $E: y^2 = x^3 - 35x + 98$ over $\mathbb{F}_{701}$, with endomorphism ring $\operatorname{End}(E) = \Z \left[ \frac{1+\sqrt{-7}}2 \right]$. Observe that the discriminant $D$ of $\End(E)$ is -7, and that $\lp \frac{-7}{5} \rp = -1$, so Theorem \ref{Thm Charles} tells us that $E[5]$ admits a distortion map. The distortion map $\phi$ for $E[5]$ given in \cite{Cha06} corresponding to multiplication by $\alpha = \frac{1+\sqrt{-7}}2$ is the rational map
\[ \phi(x,y) = \left( \alpha^{-2} \left(x- \frac{7(x-\alpha)^4}{x+\alpha^2-2}\right), \alpha^{-3}y \left( 1 + \frac{7(1-\alpha)^4}{(x+\alpha^2-2)^2}\right) \right), \]
and $\alpha$ reduces to $386 \in \mathbb{F}_{701}$. Pick the $2$-torsion point $P=(319,0)$ on $E$, then the group theoretic quotient $E\longrightarrow E/\langle P \rangle$ corresponds to a degree $2$ isogeny of elliptic curves
$$
\begin{matrix}
\psi : & E & \longrightarrow & E' \\
& (x,y) & \longmapsto & \left(   \frac{x^2 - 319x + 313}{x - 319}, \frac{x^2y + 63xy - 197y}{x^2 + 63x + 116}   \right),
\end{matrix}
$$
where $E^\prime$ is given by the equation $y^2 = x^3 + 503x + 66$ over $\mathbb{F}_{701}$. We computed this equation using SageMath, but the procedure can be found in V\'{e}lu's paper \cite{Vel71}. The dual isogeny is given by
$$
\begin{matrix}
\widehat{\psi} : & E' & \longrightarrow & E \\
& (x,y) & \longmapsto & \left(   \frac{-175x^2 - 191x - 52}{x - 63},
 \frac{263x^2y - 191xy + 84y}{x^2 - 126x - 237}    \right).
\end{matrix}
$$

Now Corollary \ref{Cor transfer} gives that $\phi^\prime := \psi \circ \phi \circ \widehat{\psi}$ is an appropriate distortion map for $E^\prime[5]$, which we also verify by computation for all points $E^\prime[5]$. For instance, pick $Q=(675,16)$ a $5$-torsion point on $E^\prime$. By computation, we have that $e_5 (Q, \phi^\prime(Q)) = 638 \neq 1$, which is a fifth root of unity in $\mathbb{F}_{701}$.


\end{example}


\section{Distortion maps for supersingular elliptic curves} \label{Section 4}

Verheul \cite[Theorem 5]{Ver04} proved the existence of distortion maps for any point $P$ of a supersingular elliptic curve. His proof is based on the following theorem of Tate, which relates the endomorphism ring of $E$ with the endomorphism ring of its Tate module $T_\ell(E)$.

\begin{theorem}[\cite{Tat66}]\label{Theorem Tate}
Let $K$ be a finite field and $E / K$ be an elliptic curve. Fix a prime $\ell \neq \characteristic(K)$. Let $G = \Gal(\overline{K}/K)$ and $\End_G(T_\ell(E))$ denote the ring of Galois-equivariant endomorphisms of $T_\ell(E)$. The canonical map
$$
\End_{K}(E) \otimes \Z_\ell \longrightarrow \End_G(T_\ell(E))
$$
is an isomorphism of rings.
\end{theorem}

We shall extend Verheul's argument to prove the existence of a distortion map for $E[m]$ for any $m$ coprime to $p$.\svs

Let $E / \F_q$ be a supersingular elliptic curve and $m \in \Z$ coprime to $p$. Given an isogeny $\phi \in \End(E)$, its restriction to $E[m]$ defines a $\Z/m\Z$-linear endomorphism of $E[m]$. This defines a ring homomorphism:
$$
\begin{matrix}
\textup{res}_m : & \End(E) & \longrightarrow & \End_{\Z/m\Z}(E[m])  \\
& \phi & \longmapsto & \phi \vert_{E[m]}
\end{matrix}.
$$

\begin{lemma}\label{lemma_surjectivity}
The map $\textup{res}_m$ is surjective.
\end{lemma}

\begin{proof}
Let $\ell \neq p$ be a prime. Let $K = \F_{q^k}$ be the minimal field of definition of $\End(E)$. Let $\sigma \in \Gal(\overline{\F}_q / \F_{q^k})$ be the $q^k$-th power Frobenius map, and $\pi_{q^k} \in \End(E)$ be the $q^k$-th power Frobenius endomorphism. Then, for any isogeny $\phi \in \End(E)$, since $\phi$ is defined over $K$, it commutes with the action of $\sigma$. Since $\sigma$ and $\pi_{q^k}$ act in the same way, $\phi$ commutes with $\pi_{q^k}$. Thus, $\pi_{q^k}$ lies in the center of $\End(E)$, \textit{i.e.}, $\pi_{q^k} \in \Z$. Consequently, $\pi_{q^k}$ also lies in the center of $\End(T_\ell(E))$. As $\pi_{q^k}$ acts on $T_\ell(E)$ as $\sigma$, then any endomorphism of $T_\ell(E)$ is $\Gal(\overline{\F}_q / \F_{q^k})$-equivariant. Thus, Theorem \ref{Theorem Tate} establishes an isomorphism
$$
\End(E)\otimes \Z_\ell \rightarrow \End(T_\ell(E)).
$$

The surjection $\End(T_\ell(E)) \onto \End_{\Z/\ell^\alpha\Z}(E[\ell^\alpha])$ induces a surjection
$$
\textup{res}_{\ell^\alpha}:\End(E) \onto \End_{\Z/\ell^\alpha\Z}(E[\ell^\alpha])
$$
for any $\alpha \geqslant 1$. Consider the prime factorization $m = \ell_1^{\alpha_1} \cdots \ell_n^{\alpha_n}$. By the Chinese remainder theorem,
$$
E[m] \iso \prod_{i=1}^n E[\ell_i^{\alpha_i}],
$$
which induces an isomorphism
\begin{equation}\label{iso endomorphisms}
\End_{\Z/m\Z}(E[m]) \iso \prod_{i=1}^n \End_{\Z/\ell_i^{\alpha_i}\Z}(E[\ell_i^{\alpha_i}]).
\end{equation}

We claim that the map
$$
\prod_{i=1}^n \textup{res}_{\ell_i^{\alpha_i}} : \End(E) \to \prod_{i=1}^n \End_{\Z/\ell_i^{\alpha_i}\Z}(E[\ell_i^{\alpha_i}])
$$
is surjective. Indeed, given $(\lambda_1, \ldots, \lambda_n)$ in the codomain, take $\phi_i \in \End(E)$ with $\textup{res}_{\ell_i^{\alpha_i}}(\phi_i) = \lambda_i$, and integers $a_i$ congruent to 1 modulo $l_i^{\alpha_i}$ and to 0 modulo $l_j^{\alpha_j}$ for $j \neq i$. Then, the isogeny $\sum_{i=1}^n a_i \phi_i$ restricts to $(\lambda_1, \ldots, \lambda_n)$. Since the isomorphism (\ref{iso endomorphisms}) respects the restriction maps, the result follows.
\end{proof}

A consequence of the previous lemma is that any point $P$ of order coprime to $p$ admits a distortion map. In fact, the following stronger statement holds:

\begin{theorem}
$E[m]$ admits a distortion map.
\end{theorem}

\begin{proof}
Fix a basis of $E[m]$, so any $\Z/m\Z$-linear endomorphism of $E[m]$ is given by a matrix $A \in \textup{M}_2(\Z/m\Z)$. We claim that $A$ has no eigenvectors $v \in (\Z/m\Z)^2 \setminus \{0\}$ if and only if
$$
\det(\lambda I - A) \in (\Z/m\Z)^\times
$$
for all $\lambda \in \Z/m\Z$. For the necessary condition, fix $\lambda \in \Z/m\Z$ and assume that $\det(\lambda I - A) \in (\Z/m\Z)^\times$. Then, $\lambda I - A$ is invertible. If $A v = \lambda v$ for some $v \in (\Z/m\Z)^2$, then $(\lambda I - A)v = 0$ and by multiplying with $(\lambda I - A)^{-1}$, we obtain that $v = 0$. This shows that $A$ has no eigenvector. For the sufficient condition, observe that, for any $\lambda \in \Z/m\Z$, the linear map
$$
\lambda I - A: (\Z/m\Z)^2 \longto (\Z/m\Z)^2
$$
is injective. Since $(\Z/m\Z)^2$ is finite, it is also surjective. If $e_1 = (1, 0)$ and $e_2 = (0,1)$, there are vectors $v_i \in (\Z/m\Z)^2$ with $(\lambda I - A)v_i = e_i$ for $i=1,2$. The matrix with columns $v_1, v_2$ is then the inverse of $\lambda I - A$, so $\det(\lambda I - A) \in (\Z/m\Z)^\times$.\svs

By Lemma \ref{lemma_surjectivity}, the statement reduces to finding $A \in \textup{M}_2(\Z/m\Z)$ for which $\det(\lambda I - A)$ is invertible for every $\lambda \in \Z/m\Z$. By the Chinese remainder theorem, it suffices to do this for $A \in M_2(\Z / \ell^\alpha \Z)$. If $\ell \neq 2$, let $a$ be a quadratic non-residue. Then, the polynomial $x^2 - a \in (\Z / \ell^\alpha \Z)[x]$ takes values in $(\Z / \ell^\alpha \Z)^\times$ for all $\lambda \in (\Z / \ell^\alpha \Z)$. Hence, the matrix
$$
A =
\begin{pmatrix}
0 & a \\ 1 & 0
\end{pmatrix}
$$
satisfies the desired property. For $\ell = 2$, one can take the matrix
$$
A =
\begin{pmatrix}
0 & 1 \\ 1 & 1
\end{pmatrix},
$$
with characteristic polynomial $x^2 - x - 1$, which also satisfies the desired property.
\end{proof}

\section{Algorithms and known distortion maps}\label{Section 5}

Having discussed properties and the question of existence of a distortion map $\phi \in \mathrm{End}(E)$ in the previous sections, we now turn to the literature to list specific algorithms for obtaining the rational coordinates of $\phi$. We will also list some known distortion maps in Tables \ref{table1} and \ref{table2}. Furthermore, we will restrict our attention to maps where $\phi$ is not the Frobenius endomorphism, as this case is well-known. The algorithms discussed will thus focus on finding distortion maps on the eigenspaces of the Frobenius endomorphism. This approach is advantageous for applications: it allows for the distortion of $\F_q$-points without requiring further field extensions.\svs

The main backbone of the algorithms for computing distortion maps on a curve $E$ is that they rely on finding a finite sequence of isogenies $\psi_i$,  $i=1,\ldots,c$ until the target space $E^\prime$ is an elliptic curve isomorphic to $E$, as follows:
\begin{equation}\label{eq:isogeny_seq}
E = E_0 \xrightarrow{\psi_1} E_1 \xrightarrow{\psi_2} \cdots \xrightarrow{\psi_{c-1}} E_{c-1} \xrightarrow{\psi_c} E_c = E^\prime \overset{\gamma}{\simeq} E.
\end{equation}

Denote $\psi:= \psi_c \circ \cdots \circ \psi_1$. Then, the candidate distortion map is defined as
\[ \phi := \gamma \circ \psi :E \rightarrow E.\]

In general, the primary computational cost of constructing $\phi$ lies in computing a suitable isogeny chain $\psi$ which will correctly yield a distortion map. In contrast, determining the isomorphism $\gamma$ is typically the easier part. The isogenies $\psi_i$ are computed using V\'{e}lu's formulas \cite{Vel71} by considering quotients of elliptic curves by finite cyclic subgroups. 
Assuming one has the information of the order $\mathrm{End}(E)$ and its discriminant $D$, the strategy for $\psi$ is then to be an isogeny of degree $|D|$, relating the intermediate isogenies $\psi_i$ with the prime divisors of $D$. 

\subsection{Algorithms for supersingular elliptic curves}

In the supersingular setting, the assumption that one understands $\mathrm{End}(E)$ is not particularly restrictive in practice, as Galbraith and Rotger explain \cite[Section 6]{GR04}. Indeed, the only known method for constructing a supersingular elliptic curve is via the Complex Multiplication method in characteristic $0$ paired with the Deuring reduction theorem. Consequently, one understands $\mathrm{End}(E)$ when constructing $E$. On the other hand, determining the endomorphism ring from the supersingular curve $E$ alone is a computationally hard problem, which typically underpins the security of cryptosystems built on isogenies of supersingular curves. Although Kohel's algorithm \cite{Koh96} can compute a basis for $\mathrm{End}(E)$, it runs in exponential time.\svs

We list below the primary references in the literature that have detailed algorithms for constructing distortion maps. From the previous discussion, these require as input the field $\mathbb{Q}(\sqrt{-d})$ by which the characteristic $0$ elliptic curve has complex multiplication, before the reduction to $\mathbb{F}_q$. As these algorithms aim to distort points $P$ from Frobenius eigenspaces, their correctness rely on $P$ not being in $\ker \left( \phi \pi_q - \pi_q \phi \right)$, where $\pi_q$ is the $\F_q$-Frobenius, see \cite[Lemma 5.1]{GR04}. Additionally, algorithms $(1),(2)$ rely also on \cite[Proposition 6.1]{GR04}. Here is the list:
\begin{enumerate}
    \item Moody's algorithm \cite[Algorithm 2]{Moo09}.
    \item Galbraith and Rotger's algorithm \cite[Algorithm 1]{GR04}.
    \item Constructive Deuring Correspondence by Eriksen et al \cite{EPSV24}.
\end{enumerate}

Despite sharing the structural approach of Equation (\ref{eq:isogeny_seq}), the algorithms employ distinct methods to find the isogeny $\psi$. Moody utilizes the Weierstrass $\wp$-function up to a certain precision alongside V\'{e}lu's formulas. It is important to note that the streamlined description of the algorithm in the journal version \cite{Moo09} may hide this fact, but the idea is described in greater detail in his PhD thesis \cite[Alg. 3]{Moo09thesis}.\svs

Galbraith and Rotger \cite[Theorem 5.2]{GR04} gave an alternative constructive proof for the existence of a distortion map for the points $P \in E[\ell]$ laying in an eigenspace of $\pi_q \mid_{E[\ell]}$ when $k(q,\ell) > 1$, which does not rely on the Tate module. Their algorithm is derived from this constructive proof. In order to arrive at a suitable $\psi$ in Equation (\ref{eq:isogeny_seq}), the algorithm performs a walk on an isogeny graph that is defined as a tree of $j$-invariants connected by prime degree isogenies. The tree is built using a separate algorithm by Galbraith \cite{Gal99}.

\begin{notation}
We denote by $B= (\alpha, \beta \mid \mathbb{Q})$ the quaternion algebra over $\mathbb{Q}$ with basis $\{1, i, j, k\}$, subject to the multiplication rules $i^2 = \alpha$, $j^2 = \beta$, and $ij = -ji = k$, where $\alpha, \beta \in \mathbb{Q}^\times$.   
\end{notation}

The goal of the constructive Deuring correspondence is to find a supersingular elliptic curve $E$ such that $\mathrm{End}(E) \cong \mathcal{O}$ for a given maximal order $\mathcal{O}$ in $B_{p, \infty}$, the unique (up to isomorphism) quaternion algebra ramified at $p$ and $\infty$. For details on the entire correspondence algorithm and its implementation, see \cite{EPSV24}. We will focus on the curves and the distortion maps it constructs.
\begin{itemize}
    \item \textbf{Case $p \equiv 3 \pmod 4$:} The curve is $E_0: y^2 = x^3 + x$. The endomorphism ring is isomorphic to $\mathbb{Z} \oplus \mathbb{Z}i \oplus \mathbb{Z}\frac{i+j}{2} \oplus \mathbb{Z}\frac{1+k}{2}$ in the algebra $B_{p,\infty} = (-1, -p \mid \mathbb{Q})$.
    \item \textbf{Case $p \equiv 2 \pmod 3$:} The curve is $E_0 : y^2 = x^3 + 1$. The endomorphism ring corresponds to the maximal order $\mathcal{O}_{0}=\mathbb{Z}\oplus\mathbb{Z}\frac{1+i}{2}\oplus\mathbb{Z}\frac{j+k}{2}\oplus\mathbb{Z}\frac{i+k}{3}$ in $B_{p,\infty} = (-3, -p \mid \mathbb{Q})$. The basis elements rely on generators satisfying $i^2 = -3$ and $j^2 = -p$, where the element $(i-1)/2$ corresponds to the order-3 automorphism of $E_0$.
    
    \item \textbf{Case $p \equiv 1 \pmod 4$:} This case is more complex. The algorithm finds the smallest prime $\alpha \equiv 3 \pmod 4$ 
    such that $p$ is inert in $\mathbb{Q}(\sqrt{-\alpha})$. It then computes the Hilbert class polynomial $H_{-\alpha}$ to find a rational root $j \in \mathbb{F}_p$, constructing $E_0$ with this $j$-invariant. An isomorphism allows mapping the provided order $\mathcal{O}$ into the specific algebra representation used by $\mathcal{O}_0$ defined as \cite[Equation (1)]{EPSV24}.
\end{itemize}

Following the approach in \cite{EPSV24}, for primes $p \equiv 1 \pmod 4$, one can construct a supersingular elliptic curve $E_0$ such that its endomorphism ring contains an explicitly known endomorphism $\phi$ suitable for distortion. The construction proceeds as follows.
\begin{enumerate}
    \item Find the smallest prime $\alpha \equiv 3 \pmod 4$ such that $p$ remains inert in $\mathbb{Q}(\sqrt{-\alpha})$ (\textit{i.e.}, $\lp \frac{-\alpha}{p} \rp = -1$).
    \item Compute the Hilbert class polynomial $H_{-\alpha}(X)$ and find a rational root $j_0 \in \mathbb{F}_{p}$. Construct $E_0$ with $j(E_0) = j_0$.
    \item The endomorphism ring $\mathrm{End}(E_0)$ is isomorphic to a maximal order $\mathcal{O}_0$ in the quaternion algebra $B_{p, \infty} = (-\alpha, -p \mid \mathbb{Q})$. The algebra is generated by $1, i, j, k$ with $i^2 = -\alpha$, $j^2 = -p$, and $ij = -ji = k$.
    \item The distortion map $\phi$ corresponds to the element $\sqrt{-\alpha}$ in the algebra. It satisfies $\phi^2 = [-\alpha]$ and $\phi \pi = -\pi \phi$ where $\pi$ is the Frobenius.
\end{enumerate}

Computationally, $\phi$ is evaluated by composing a normalized isogeny of degree $\alpha$ with an isomorphism. If $E_0: y^2 = x^3 + ax + b$, one defines the auxiliary curve $E'_0: y^2 = x^3 + \alpha^2ax - q^3b$ and the isomorphism $\tau: E_0 \to E'_0$ defined by $(x,y) \mapsto (-\alpha x, \sqrt{-\alpha}^3 y)$.\svs

The composition $\phi' = \tau \circ \phi$ is a normalized isogeny $E_0 \to E'_0$ of degree $\alpha$, which can be computed efficiently by \cite{BMSS08}. The target endomorphism is then recovered as
\[ \phi = \tau^{-1} \circ \phi^\prime. \]

We list in Table \ref{table1} some known distortion maps for supersingular elliptic curves together with their respective reference in the literature. The three cases from the constructive Deuring correspondence are included. See also \cite{Moo09thesis} for nine specific curves and distortion maps with respect to the nine class number $1$ imaginary quadratic fields, as well as \cite{GM05} for more examples.
{\begin{table}[h]
\centering 
\caption{Distortions in some supersingular curves with respect to the torsion points defined over the ground field on the first column, for $p>3$.}\label{table1}
\setlength{\tabcolsep}{1.5pt}\resizebox{.97\textwidth}{!}{
\begin{tabular}{c c c c c c}
Field & Curve & Morphism & Conditions & Group order & Reference \\
\hline
\multirow{2}{*}{$\mathbb{F}_{p}$} & \multirow{2}{*}{$y^{2}=x^{3}+ax$} &
$(x,y)\mapsto(-x,iy)$ 
& \multirow{2}{*}{$p\equiv 3\pmod{4}$} & \multirow{2}{*}{$p + 1$} & \cite{JN03}, \\
& & $i^{2}=-1$ & & & \cite{EPSV24} \\
\\
\multirow{2}{*}{$\mathbb{F}_{p}$} & \multirow{2}{*}{$y^{2}=x^{3}+a$} & $(x,y)\mapsto(\zeta x,y)$ 
& \multirow{2}{*}{$p\equiv 2\pmod{3}$} & \multirow{2}{*}{$p + 1$} & \cite{JN03}, \\
& & $\zeta^{3}=1$ & & & \cite{EPSV24} \\
\\
\multirow{3}{*}{$\mathbb{F}_{p^{2}}$} & \multirow{2}{*}{$y^{2}=x^{3}+a$} & $(x,y)\mapsto\left(\omega\frac{x^{p}}{r^{(2p-1)/3}},\frac{y^{p}}{r^{p-1}}\right)$ & \multirow{3}{*}{$p\equiv 2\pmod{3}$} & \multirow{3}{*}{$p^{2}-p+1$} & \multirow{3}{*}{\cite{JN03}} \\
& \multirow{2}{*}{$a\not\in\mathbb{F}_{p}$} & $r^{2}=a,r\in\mathbb{F}_{p^{2}}$ & & & \\
& & $\omega^{3}=r,\omega\in\mathbb{F}_{p^{6}}$ & & & \\
\\
\multirow{3}{*}{$\mathbb{F}_p$} & \multirow{2}{*}{$y^2 = x^3 + ax + b$} & \multirow{2}{*}{$\phi^2 = [-q],$} & $p \equiv 1 \pmod 4,$ & \multirow{3}{*}{$p+1$} & \multirow{3}{*}{\cite{EPSV24}} \\
& \multirow{2}{*}{(via root of $H_{-\alpha}$)} &  \multirow{2}{*}{$\phi \pi = -\pi \phi$} &  $\alpha \equiv 3 \pmod 4,$  & & \\
& & & minimal $\lp \frac{-\alpha}{p} \rp = -1$ & & \\
\\
\multirow{3}{*}{$\F_{2^d} $} & \multirow{3}{*}{$y^2 + y = x^3 + x$} & $(x,y)\mapsto (x+s^2,y+sx+t)$ & & \multirow{3}{*}{$2^d + 1 \pm 2^{(d+1)/2}$} & \multirow{2}{*}{\cite{Moo09thesis},} \\
& & $s^2+s+1 =0, $& & & \multirow{2}{*}{\cite{GR04}} \\
& & $t^2 + t + s = 0$ & & & \\
\\
\multirow{3}{*}{$\F_{3^d}$} & \multirow{3}{*}{$y^2 = x^3 + x + 1$} & $(x,y)\mapsto (-x+s, iy)$ & & \multirow{3}{*}{$3^d + 1 \pm 3^{(d+1)/2}$} & \multirow{2}{*}{\cite{Moo09thesis},} \\
& & $s^3+2s+2 = 0, $& & & \multirow{2}{*}{\cite{GR04}} \\
& & $i^2 = -1$ & & & \\
\hline
\end{tabular}}
\end{table}}


\subsection{Algorithms for ordinary elliptic curves} In $2009$, Moody \cite{Moo09} stated that it would be interesting to have algorithms that efficiently compute distortion maps for ordinary elliptic curves in the case of $k(q,\ell)=1$. To date, no such algorithm appears to exist in the literature; however, recent developments offer some new perspectives, which we discuss below.\svs

In the work of Fouquet et al. \cite{FM18}, distortion maps from isogenies as in Equation (\ref{eq:isogeny_seq}) are modeled as closed walks on the crater of an isogeny volcano (for background on isogeny volcanoes, one may also consult \cite{Koh96,BCP24} for instance). Specifically, consider an ordinary curve $E$ and primes $m,\ell$. Let $\phi = \gamma \circ \psi$ arise as in Equation (\ref{eq:isogeny_seq}) on the crater of the $m$-volcano, and let $P$ be an $\ell$-torsion point that yields a descending isogeny to a curve $E^\prime$ (within the $\ell$-volcano). Under these conditions, $\phi$ fails to distort $P$ if and only if the $m$-craters of $E$ and $E^\prime$ have the same cardinality \cite[Proposition 5.1]{FM18}. While this approach does not offer a deterministic algorithm yet, it effectively serves as a volcano analogue to Galbraith and Rotger's strategy of isogeny walks in the supersingular case.\svs

In \cite{HWXZ13}, Hu et al. provide an efficient algorithm for generating ordinary elliptic curves of embedding degree $1$ via the Complex Multiplication method. Furthermore, their Theorem 2 establishes that the symmetric reduced Tate pairing $t(P,P)$ is non-degenerate on the curves produced by their Algorithm 1. Depending on the Legendre symbol (as detailed in our Proposition \ref{Proposition Legendre symbol}), this non-degeneracy holds either for all non-trivial points in $E[r]$, where $r$ is prime or for all points excluding two specific eigenspaces. Their construction ensures that $r$ is coprime to the elliptic curve's discriminant, ruling out the single eigenspace case.
Thus, the reduced Tate pairing can be used instead of the modified Weil pairing without the cost of explicitly evaluating a distortion map $\phi$. However, assuming $\mathrm{End}(E)=\Z[\phi]$, the proof of Theorem 2 reveals that the non-degeneracy is intrinsically linked to the underlying behavior of $\phi$: the reduced Tate pairing is non-degenerate precisely when $\phi$ acts as a distortion map on $P$.\svs

Furthermore, Hu et al. \cite{HXZ11} present two classes of ordinary elliptic curves that admit computable distortion maps. This result, formulated as Theorem $6$, provides a partial generalization of Verheul's theorem \cite[Thm. 3]{Ver04} for ordinary curves, regarding the hardness of the DLP on $\F_p^*$. The necessary conditions and a table of these distortion maps are listed below.
\begin{itemize} 
\item \textbf{Type I:} Let $\omega = (1+\sqrt{-3})/2$, and $p,r$ be primes such that $r\equiv 2\pmod 3$ and $p = r^2 + r + 1$. Let $\nu = \omega - r\overline{\omega}$ and let $(\cdot)_6$ be the $6$th power residue symbol taking values in $\{1,-1,\omega,-\omega,\omega^2,-\omega^2\}$. Assume for the coefficient $b$ that $(\frac{4b}{\nu})_6 = -\omega$.
\item \textbf{Type II:} Let $r\equiv 3\pmod 4$ and $p=16r^2 + 1$ be primes.
\end{itemize}

\begin{table}[h]
\centering 
\caption{Distortion maps for ordinary curves with respect to the entire group $E(\F_p)$.}\label{table2}
\setlength{\tabcolsep}{4pt}
\begin{tabular}{c c c c c c } 
Field & Curve & Morphism & Conditions & Group Order & Reference \\
\hline
$\mathbb{F}_{p}$ & $y^{2}=x^{3}+b$ & $(x,y)\mapsto(r\cdot x,y)$ & Type I & $r^2$&  \cite{HXZ11} \\
 $\F_p$&$y^2 = x^3 - x$ & $(x,y)\mapsto (-x,4r \cdot y)$& Type II & $16r^2$& \cite{HXZ11}  \\
\end{tabular}
\end{table}

There are a few more examples, one may for instance use ordinary reduction on the curves in \cite[Prop. 2.3.1]{Sil94} as Charles did in \cite[Example 3.2]{Cha06}.

\begin{remark}
    In the case of $\ell$ dividing the conductor, as in Example \ref{example_conductor}, what worked as a strategy for a distortion map $\phi$ on some points of $E[\ell]$ was to pick a point of order $\ell^2$ for the associated isogeny in (\ref{eq:isogeny_seq}). This resulted in $\phi^2 = 0$ on $E[\ell]$. We can showcase this similarly, in \cite[Example 3.4]{Cha06}, contradicting the claim that $\mathrm{End}(E)\pmod2 \cong \Z/2\Z$. The ordinary curve is $y^2 = x^3 + 11x + 4$ over $\F_{13}$. Let $\F_{13^6} = \F_{13}[X]/(X^6 + 10X^3 + 11X^2 + 11X + 2)$ generated by $z$. The $4$-torsion point $P=(4, 3z^5 + 9z^4 + 7z^3 + 12z + 5)$ in $E(\F_{13^6})$ yields an isogeny $\psi:E \rightarrow E^\prime : y^2 = x^3 + 8x + 9$. There is an isomorphism $\gamma :E^\prime \rightarrow E$ given by $(x,y)\mapsto (4x,-5y)$. Then, $\phi := \gamma \circ \psi $ distorts the $2$-torsion points $(6,0),(9,0)$ by mapping them to $(11,0)$, and maps $(11,0)$ to the point at infinity.
\end{remark}

\section{Distortion maps in higher dimension} \label{Section 6}

For higher dimensional abelian varieties in general (see \cite{Mil86}), there is also a notion of a distortion map which has the same properties as in the one-dimensional case. If $A$ is an abelian variety and $A^{\vee}$ denotes its dual, the Weil pairing is a bilinear map
$$e_m: A[m] \times A^{\vee}[m] \rightarrow \mu_m,$$
and if we are given a principal polarization $\psi: A\longrightarrow A^{\vee}$, we may define:
$$e_m^{\psi}: A[m]\times A[m]\rightarrow \mu_m,$$ by the formula $e_m^{\psi}(P,Q)=e_m(P,\psi(Q))$, for $P, Q \in{A[m]}$.

In the case of jacobians of higher genus curves, the generalization of distortion maps is for instance developed in the work \cite{GPRS09}. Let us consider $C$ to be a curve of genus $g\geqslant 1$. Its jacobian $\Jac(C)$ is an abelian variety of dimension $g$. The Abel-Jacobi map provides an isomorphism $\psi: \Jac(C) \xrightarrow{\sim} \Jac(C)^{\vee}$, in other words a principal polarization for $\Jac(C)$.\svs

Similarly to the elliptic curve case, there is a notion of a distortion map for a jacobian of a higher genus curve:
\begin{definition}
    Let $m\geqslant 1$ be a natural integer. Let $D \in \Jac(C)[m]$. An endomorphism $\varphi \in \End(\Jac(C))$ is called a \textit{distortion map for} $D$ if $e_m^{\psi}(D, \varphi(D))\neq 1$.
\end{definition}

In the article \cite{GPRS09} the authors focus on curves with supersingular jacobians:
\begin{definition}
    An abelian variety $A$ of dimension $g\geqslant 1$ is called \textit{supersingular} if it is isogenous (over $\overline{\mathbb{F}}_q$) to $E^g$ for some supersingular elliptic curve $E$.
\end{definition}

\begin{remark}
    In dimension $g$, the subgroup $A[p]$ where $p$ is the characteristic of a basefield is an $\mathbb{F}_p$-module of rank $0 \leqslant r \leqslant g$. A supersingular abelian variety always has $p$-rank $r=0$, and an ordinary abelian variety has $p$-rank $r=g$, but there are also intermediate cases when $g\geqslant 2$. In the case of elliptic curves (dimension $g=1$), there are no intermediate cases, and in the ordinary case it is always true that $E(\overline{\mathbb{F}}_p)[p] \cong \Z/p\Z$.
\end{remark}

\begin{remark}
    One can show that in the supersingular case $\End(A)\otimes \mathbb{Q}$ is a central simple algebra over $\mathbb{Q}$ of dimension $(2g)^2$.
\end{remark}

Considering that Tate's Theorem \ref{Theorem Tate} (from \cite{Tat66}) is originally proved for general abelian varieties over finite fields, the argument used in the proof of Lemma \ref{lemma_surjectivity} can be adapted to the higher dimension case to obtain the following result:
\begin{theorem}[\cite{GPRS09}, Theorem 2.1]
    Let $A$ be a supersingular abelian variety of dimension $g$ over $\mathbb{F}_q$,
and let $r$ be a prime not equal to the characteristic of $\mathbb{F}_q$. Assume $\psi: A\rightarrow A^{\vee}$ is a principal polarization. For every non-trivial
element $D \in A(\mathbb{F}_q)[r]$, there exists $\phi \in \End(A)$ such that
$e_r^{\psi}(D, \phi(D)) \neq 1$.
\end{theorem}

In \cite{RS02} it is shown that distortion maps in genus $\geqslant3$ have no interest for applications in cryptography. The case of genus two remains interesting, and supersingular jacobian surfaces are studied in \cite{GPRS09} by a case-by-case analysis. In particular, the article \cite{GPRS09} provides constructive ways to obtain a distortion map in the case where the embedding degree is $\geqslant 4$.

\printbibliography

\end{document}